\documentclass{amsart}
\usepackage{verbatim,amssymb,amsmath,amscd,latexsym,amsbsy,mathrsfs} 
\usepackage{graphicx}
\usepackage[all]{xy}
 \input{xy}
 \xyoption{all}
\usepackage{color}
\usepackage{xparse}
\usepackage{mathtools}
\usepackage[colorlinks = true,
            linkcolor = blue,
            urlcolor  = green,
            citecolor = green,
            anchorcolor = blue]{hyperref}   

\newtheorem{thm}{Theorem}[section]
\newtheorem{cor}[thm]{Corollary}
\newtheorem{df}[thm]{Definition}
\newtheorem{prop}[thm]{Proposition}

\newtheorem{lemma}[thm]{Lemma}
\newtheorem{rmk}[thm]{Remark}

\newtheorem{ex}[thm]{Example}

\DeclareMathOperator*{\rank}{rank}

\DeclareMathOperator*{\DR}{DR }

\DeclareMathOperator*{\parc}{par-c} 
 
\DeclareMathOperator* {\Comp}{Comp}

\newcommand{\sslash}{\mathbin{/\mkern-6mu/}}

\newcommand{\NN}{\mathbb{N}}

\newcommand{\ZZ}{\mathbb{Z}}

\newcommand {\C} {{\mathbb C}}
\newcommand {\R} {{\mathbb R}}
\newcommand {\Z} {{\mathbb Z}}
\newcommand {\Q} {{\mathbb Q}}

\newcommand {\E} {{\mathcal E}}
\newcommand {\dt} {{\bullet}}

\newcommand {\OO} {{\mathcal O}}

\newcommand{\pullback}[1]{{#1}^*}
\newcommand{\pushforward}[1]{{#1}_*}
\newcommand{\res}{\text{Res}}
\newcommand{\Gr}{\text{Gr}}
\newcommand{\parabolic}{(\pushforward{\pi}\tilde V\otimes
  O_Y(\floor{-\alpha N}\tilde D))^G}

\DeclarePairedDelimiter\floor{\lfloor}{\rfloor}

\begin{document}
\title{Vanishing theorems for parabolic Higgs bundles}
 \author{ Donu Arapura}
  \author{Feng Hao}
 \author{ Hongshan Li}
\address{Department of Mathematics\\
  Purdue University\\
  West Lafayette, IN 47907\\
  U.S.A.}

\begin{abstract} 
The main result is a Kodaira  vanishing theorem for semistable
parabolic Higgs bundles with trivial parabolic  Chern classes. This
implies a general semipositivity theorem. This also implies
a Kodaira-Saito vanishing theorem for complex variations of Hodge structure.
\end{abstract}
\maketitle

\section*{Introduction}

Let $X$ be a complex smooth projective variety  and
$D\subset X$ a reduced divisor  with simple normal crossings. In an earlier paper
\cite{arapura}, the first author  reproved (and slightly extended) Saito's Kodaira vanishing theorem
for (complex) polarized variations of Hodge structures on $X-D$ with unipotent
monodromy around $D$, by deducing it from a more general vanishing
theorem for Higgs bundles. In this paper, our goal is to extend
this further to complex polarized variations of Hodge structure without any unipotency
condition.  By work of Mochizuki and Simpson, such a variation determines a parabolic Higgs bundle,
which  consists of a vector bundle $E$ on $X$,  a map $\theta:E\to
\Omega_X^1(\log D)\otimes E$ such that $\theta\wedge \theta=0$ and 
an $\R$-indexed filtration on $E^*\subset E(*D)$ satisfying appropriate conditions.
The first condition tells us that we can form a  ``de Rham'' complex
$$\DR(E,\theta):=E\stackrel{\theta}{\to}\Omega_X^1(\log D)\otimes E\stackrel{\theta}{\to}\Omega_X^2(\log D)\otimes E\stackrel{\theta}{\to}\ldots$$
For parabolic Higgs bundles coming from complex variations, this is just
the Kodaira-Spencer complex. Also in this case, sections of
$E|_{X-D}$ lie in $E^\alpha$ if their norms, with respect to the Hodge
metric, are $O(|f(x)|^{\alpha-\epsilon})$, where $f$ is a local
equation for $D$. The filtration is also related to the 
monodromy about $D$; in particular, it would be trivial  in the
unipotent case, but not otherwise.   One reason for keeping track of
this filtration is that  it enters into natural modifications of Chern
classes et cetera, where the usual formulas need to be corrected along $D$.
These notions will be reviewed in the paper.

 Our main result  (corollary \ref{cor:main}) is as follows:  Given a
slope semistable parabolic Higgs bundle $(E, E^*,\theta)$ with trivial parabolic Chern
classes, and an ample line bundle $L$,
$$\mathbb{H}^i(X, \DR(E,\theta)\otimes L)=0$$
for $i>\dim X$. More generally, theorem~\ref{stablevan} gives an extension to this in the
spirit of the Kawamata-Viehweg and Le Potier vanishing theorems.
In theorem~\ref{thm:semipos}, we show that    if  there is a decomposition $E=E_+\oplus E_-$ 
  such that $\theta(E)\subseteq \Omega_X^1(\log D)\otimes E_-$, then
  $E_+$ is nef. This is deduced using the vanishing theorem.
The assumptions of these results, and therefore their conclusions,  hold for Higgs bundles
arising from complex variations of Hodge structure.   In particular,
we recover the well known semipositivity results of Fujita, Kawamata
and many others.

A special case of the main theorem where the filtration is trivial and
$\theta$ is nilpotent was proved in \cite{arapura}. (This was proved
using characteristic $p$ methods, but a trick used here leads
in princple to a characteristic $0$ proof, when $L$ is an ample line
bundle.  See remark \ref{rmk:reductiontoVHS}.)
The proof of the  more general theorem is by  reducing it to this special case.
The reduction is done in stages, and the sketch is probably clearer if 
we describe the process  in reverse from general to special.
Using an approximation argument, we reduce to the case where the
weights, which are the numbers where  the filtration $E^*$ jumps, are rational.  Yokogawa \cite{yokogawa1} shows that there is a
moduli space of  parabolic Higgs bundles with  fixed rational weights
and vanishing parabolic Chern classes which are semistable in the
appropriate sense. Then using the natural $\C^*$-action on this space and upper
semicontinuity of cohomology, we reduce to the case where $\theta$ is nilpotent.
Again using the rationality of the weights, a result of  Biswas
\cite{biswas} shows there exists a branched covering $\pi:Y\to X$ and a
nilpotent semistable Higgs bundle $(\mathcal{E}, \vartheta)$ on $Y$,
with trivial filtration, such
that $\mathbb{H}^i(X, \DR(E,\theta)\otimes L)$ is a summand of
$\mathbb{H}^i(Y, \DR (\mathcal{E}, \vartheta)\otimes \pi^*L)$.
This  is zero by the special case.

\section{Parabolic bundles}

Let $X$ be a complex smooth projective variety with a reduced simple
normal crossing divisor $D=\sum_{i=1}^n D_i$. Let $j:U=X-D\to X$
denote the inclusion of the complement.  We fix this notation
throughout the paper.  We use the following definition, which is
equivalent to the one given by Maruyama and Yokogawa \cite{my}, although
different notationally.

\begin{df}
  A parabolic sheaf on $(X, D)$ is a torsion free
  $\mathcal{O}_X$-module $E$, together with a decreasing $\R$-indexed
  filtration by coherent subsheaves such that
  \begin{enumerate}
  \item[P1.] $E^0=E$.
  \item[P2.] $E^{\alpha+1} = E^\alpha(-D)$.
  \item[P3.] $E^{\alpha-c} = E^{\alpha}$ for any $0<c\ll 1$.
  \item[P4.] The subset of $\alpha$ such that $Gr^\alpha E\not= 0$ is
    discrete in $\R$. Here
    $Gr^\alpha E:= E^{\alpha}/E^{\alpha+\epsilon}$ for
    $0<\epsilon\ll 1$.
  \end{enumerate}
\end{df}
We refer to the filtration as a {\em parabolic structure}.  The numbers
$\alpha$ such that $Gr^\alpha E\not= 0$ are called {\em weights}. A weight
is normalized if it lies in $[0,1)$. The axioms imply that the ordered
set of positive normalized weights
$0< \alpha_1<\alpha_2<...<\alpha_\ell<1$ together with the reindexed
filtration
$$E=F^0(E)\supsetneq F^1(E)= E^{\alpha_1}\supsetneq  F^2(E)=
E^{\alpha_2}\ldots \supsetneq F^{\ell+1}(E)=E(-D)$$ determines the
whole parabolic structure. We refer to the last filtration as a
{\em quasi-parabolic structure}. Thus a parabolic structure consists of a
quasi-parabolic structure together with a choice of normalized
weights.  In certain situations, we will need to perturb the weights.

\begin{df}
  Given $\epsilon>0$, we say that two parabolic sheaves $E^*$ and
  ${E'}^*$ are $\epsilon$-close if the underlying sheaves with
  quasi-parabolic structures are isomorphic, and the normalized
  weights satisfy $|\alpha_i-\alpha_i'|<\epsilon$.
\end{df}

\begin{df}\label{df:paravect} A parabolic bundle on $(X,D)$ consists of a
  vector bundle $E$ on $X$ with a parabolic structure, such that as a
  filtered bundle, $E$ is Zariski locally a sum of rank one
  bundles. (See the discussion after example \ref{ex:paraline} for
  further explanation).
\end{df}

Many authors  use a weaker definition.
However, we have followed Iyer and Simpson \cite{is} in adopting what they call
a locally abelian parabolic bundle as our definition.   Certain
notions and constructions  given later  (weight vectors, Biswas' correspondence) become more straightforward
with this definition.  We have, however, retained the original scalar
indexing from \cite{my}, which is more convenient for our purposes.

We describe a few basic examples.

\begin{ex}\label{ex:trivial}
  Any vector bundle $E$ can be given a parabolic structure with a
  single weight $0$ and $E^{i}= E(-iD)$. We refer to this
  as the trivial parabolic structure.
\end{ex}

\begin{ex}\label{ex:paraline}
  Given any line bundle $L$ and coefficients $\beta_i \in [0,1)$ for
  each component $D_i$ of $D$, we have the following parabolic line
  bundle
  \begin{equation}
    \label{eq:L}
    L^{\alpha}:= L(\sum -\lfloor 1+ \alpha-\beta_i\rfloor D_i)  
  \end{equation}
  \noindent where $\lfloor\cdot \rfloor$ is the floor function.
\end{ex}

We will see shortly that the weights are exactly the $\beta_i$. We can
assemble these into a vector $(\beta_1,\ldots, \beta_n)\in \R^n$ that
we call the {\em normalized weight vector} for $L^*$.  We can make this
independent of the labeling by viewing it as an element of
$\mathrm{Hom}_{Set}(\Comp(D), \R)$, where $\Comp(D)$ is the set of irreducible
components of $D$.  We can recover the normalized weight vector from
the parabolic structure alone: the $i$th component of the weight
vector is $\beta_i$ if and only if $Gr^{\beta_i}L$ is nonzero at the
generic point of $D_i$.  It follows easily that any parabolic line
bundle is isomorphic to the one above for some unique normalized
weight vector.  Our definition says that Zariski locally a parabolic
bundle is a direct sum of parabolic line bundles.  It would equivalent
to formulate this in the analytic topology. The proof is implicit in
the argument given the first paragraph of \cite[p 361]{is}.

We will determine the weights and quasi-parabolic structure for the
above example. To simplify the notation, reindex the $D_i$, so that
$0\le \beta_1\le \ldots \le \beta_n <1$.  The following is
straightforward.

\begin{lemma}\label{lemma:paraline}
  The set of normalized weights is exactly the set $\{\beta_i\}$. If
  we list the weights union $0$ in increasing order
  $0=\beta_{r_0} <\beta_{r_1}\ldots <\beta_{r_n}$, then
 $$F^i(L)= L(-D_1-\ldots- D_{r_i})$$ 
\end{lemma}

We want to extend the notion of normalized weight vectors to parabolic
bundles.  Given a Zariski open $U\subseteq X$, we have a restriction
$\mathrm{Hom}(\Comp(D), \R)\to \mathrm{Hom}(\Comp(U\cap D), \R)$.  Suppose that we are
given a Zariski open cover $\{U_i\}$ of $X$ such that each $E|_{U_i}$
is a sum of parabolic line bundles. We say that $\beta$ is a
{\em normalized weight vector} of $E$ if for each $i$, $\beta|_{U_i}$ is a
normalized weight vector of a line bundle summand of $E|_{U_i}$. This
notion is easily seen to be independent of the cover.

\begin{ex}\label{ex:paraDel} 
  Suppose that $(V_o,\nabla_o)$ is a vector bundle with an integrable
  connection with regular singularities over $U$. By Deligne
  \cite{deligne}, for each $V^\alpha$ there exists a unique extension
$$\nabla^\alpha:V^\alpha\to \Omega_{X}^1(\log D)\otimes V^\alpha$$
with the eigenvalues of the residue
$Res_{D_i}(\nabla^{\alpha})\in \textnormal{End}(V^{\alpha}\otimes
\mathcal{O}_{D_i})$ having real parts in $[ \alpha,1+\alpha)$, for
each irreducible component $D_i$ of $D$. This again forms a parabolic
bundle, that we refer to as the Deligne parabolic bundle.  If the
monodromy of $\nabla_o$ around each component of $D$ is unipotent,
then $V^*$ is trivial.
\end{ex}

\begin{df}
  A parabolic Higgs sheaf or bundle on $(X,D)$ is a parabolic sheaf or
  bundle $E^*$ together with a holomorphic map
$$\theta:E\to \Omega_X^1(\log D)\otimes E$$
such that
$$\theta\wedge \theta=0$$
and
$$\theta(E^\alpha)\subseteq \Omega_X^1(\log D)\otimes E^\alpha$$\\
\end{df}

Natural examples come from variations of Hodge structure. These will
be discussed in more detail in section \ref{sect:nonabelian}.

\section{Biswas' correspondence}\label{sect:biswas}
We will assume in this section that the weights are rational with
denominator dividing a fixed positive integer $N$.  Recall that
Kawamata \cite[theorem 17]{kawamata} has constructed a smooth
projective variety $Y$, and a finite map $\pi:Y\to X$, such
that $\tilde{D}:=(\pi^*D)_{red}$ is a simple normal crossing
divisor, and $\pi^*D_i = k_iN ( \tilde D_i)$ for some integer $k_i>0$,
where $\tilde D_i=(\pi^*D_i)_{red}$.  By construction $\pi:Y\to X$ is
given by a tower of cyclic covers, so it is Galois. 
Let $G$ denote the Galois group.
A {\em $G$-equivariant vector bundle} on $Y$, is a bundle
$p:V\to Y$ (viewed geometrically rather than as a sheaf) on which $G$
acts compatibly with $p$.

We list some basic classes of examples.

\begin{ex}\label{ex:pullbackVB}
  For the above Galois covering $\pi: Y\to X$ with Galois group $G$
  and any vector bundle $V$ over $X$, $\pi^*V$ can be made into a
  $G$-equivariant bundle, so that the projections $p$
$$
\xymatrix{
  \pi^*V\ar[r]\ar[d]^{p} & V\ar[d]^{p} \\
  Y\ar[r]^{\pi} & X }
$$
are compatible with the $G$-action. Fix any point $y\in Y$, then the
action of isotropy subgroup of the point $y$ on the fiber
$(\pi^*V)_{y}$ is trivial.
\end{ex}

\begin{ex}\label{ex:equiLB}
  With the same notation as the above, the line bundle
  $\OO_Y(\tilde D_i)$ has an equivariant structure compatible with the
  one on $\pi^*\OO_X(D_i)$ under the isomorphism
  $\OO_Y(\tilde D_i)^{\otimes k_iN}\cong \pi^* \OO_X(D_i)$.
\end{ex}

There is a Higgs version of $G$-equivariant bundle given by Biswas
\cite{biswas2}.   A {\em $G$-equivariant Higgs bundle} is a pair $(V, \theta)$, with a
  $G$-equivariant bundle $V$ on $Y$ and an equivariant morphism
  $\theta: V\to \Omega_Y^1(\log D)\otimes V$, such that
  $\theta\wedge\theta=0$.
Then we have the following results given by Biswas.

\begin{thm}[Biswas {\cite{biswas}} {\cite[theorem
    5.5]{biswas2}}]\label{thm:biswas} With the notation as above, we
  have the following two equivalences of categories:

  \begin{enumerate}
  \item An equivalence $E^*\mapsto \mathcal{E}$ between the category
    of parabolic bundles on $X$ with weights in $\frac{1}{N}\Z$ and
    $G$-equivariant bundles on $Y$.

  \item An equivalence
    $(E^{*}, \theta)\mapsto (\mathcal{E}, \vartheta)$ between the
    category of parabolic Higgs bundles on $X$ with weights in
    $\frac{1}{N}\Z$ and $G$-equivariant Higgs bundles on $Y$.

  \end{enumerate}

\end{thm}

We recall the construction in one direction for (1). Given an
$G$-equivariant bundle $\mathcal{E}$ on $Y$, we obtain a parabolic
bundle
\begin{equation}\label{eq:Ealpha}
  E^\alpha = (\pi_*(\mathcal{E}\otimes \OO_Y(\lfloor -\alpha\cdot \pi^*D\rfloor))^G
\end{equation}
where
$ \lfloor -\alpha\cdot \pi^*D\rfloor=\sum_i \lfloor
-\alpha_i k_iN\rfloor \tilde D_i$.

Suppose that $(V_o,\nabla_o)$ is a vector bundle with connection
satisfying the assumptions of example \ref{ex:paraDel}. In addition,
suppose that the eigenvalues of the monodromy around $D$ are $N$th
roots of unity. Then the weights of the Deligne parabolic bundle lie
in $\frac{1}{N}\Z$.  Furthermore
$(\tilde V_o,\Box_o):=(\pi^*V_o,\pi^*\nabla_o)$ has unipotent local
monodromies.  Let $(V,\nabla)$ and $(\tilde V,\Box )$ denote
Deligne's extensions of $V_o$ and $\tilde V_o$.  The functoriallity of
this construction \cite[proposition 5.4]{deligne}, shows that $\tilde V$ is
equivariant.

\begin{lemma}
  Biswas' construction applied to $\tilde V$ yields $V^*$.
\end{lemma}

\begin{proof}
  Since we will not need this result, we will merely outline the proof
  when $\dim X=1$.  Working locally, we may assume that $V^*$ is
  a line bundle on $X$.  Suppose the residue of $\nabla$ at $D$ is
  given by $\beta$. Let
  \[
    \pi: Y \rightarrow X
  \]
  by the cyclic cover of degree $N$ branched over $D$ such that
  $\beta \in \frac{1}{N}\ZZ$.  Then, $\pi$ is locally given by
  $y^N = x$, where $x$ and $y$ are local coordinates defined on
  coordinate neighborhoods $U\subset X$ and $W\subset Y$.  We assume
  that $D$ and $\tilde D = (\pullback{\pi}D)_{\text{red}}$ are defined
  by $x=0$ and $y=0$. Let $G\cong \Z/N\Z$ be the Galois group of
  $\pi$, and let $\mu$ denote a generator.

  Let $e$ be a local frame of $V^0$ such that $\nabla^0$ is given by
  the connection matrix
  \[
    \beta\frac{dx}{x}
  \]
  Let $j \in \NN$ be an integer such that
  \[
    \frac{j}{N} = \beta
  \]
  Then, the connection matrix of $\pullback{\pi}V^0$ locally on $W$
  will be given by
  \[
    j\frac{dy}{y}
  \]
  with respect to the frame $s = \pullback{\pi}e$.

  Let $W^* = W - \tilde D$. We have an inclusion of bundles on
  $Y - \tilde D$
  \[
    \phi : \pullback{\pi}V_o \hookrightarrow \tilde V|_{Y-\tilde D}
  \]
  which extends to an isomorphism
  \[
    \phi: \pullback{\pi}V^0(j\tilde D) \rightarrow \tilde V
  \]
  Locally on $W$, it is given by
  \begin{align*}
    \phi: O_W\cdot s & \rightarrow O_W \cdot s \\
    y^{-j}s & \mapsto s
  \end{align*}  
  Now, $\pullback{\pi}V^*(j\tilde D)$ has a natural $G$-action, and the
  above isomorphism respect this action.  Locally on $W$, the action
  of $G$ on $\tilde V$ can be described as
  \[
    \mu\cdot s = \mu^{-j}\times s
  \]

  We claim that
  \[
    (\pushforward{\pi}\tilde V\otimes O_Y(\floor{-\alpha
      N}\tilde D))^G
  \]
  is the extension of $V_o$ whose residue lies in
  $[\alpha, 1 {+}\alpha)$.

  Since $(\pushforward{\pi}\pullback{\pi}V_o)^G = V_o$, We have the
  natural inclusion
  \[
    V_o \hookrightarrow (\pushforward{\pi}\tilde V)^G
  \]
  on $X - D$.

  Next, let $\nabla^{\alpha}$ be the connection on $\parabolic$, we
  compute
  $
    \res^{D}\nabla^{\alpha}
  $.
  Locally on $U$, $y^{j-\floor{-\alpha N}}s$ is a basis
  for $(\pushforward{\pi}\tilde V)^G$. So
  \[
    \res_{D}\nabla^{\alpha} = \frac{j - \floor{-\alpha
        N}}{N}
  \]
  which lies in $[\alpha, 1 {+} \alpha)$, as claimed. This
  proves the lemma.

\end{proof}

\section{Parabolic Chern classes}\label{sect:chern}

We  give a quick definiton of parabolic Chern classes.  Given the
polynomial
$$\prod_1^r (1+t(x_i+y_i))$$
we may write the coefficient of $t^k$ as a polynomial
$P_k(s_1,\ldots, s_r, y_1\ldots, y_r)$ in the elementary symmetric
polynomials $s_i= s_i(x_1,\ldots, x_r)$ and the remaining variables
$y_j$. Given a rank $r$ parabolic bundle $E^*$ with normalized weight
vectors
$\alpha^{(1)}=(\alpha^{(1)}_i),\ldots, \alpha^{(r)}
=(\alpha^{(r)}_i)$, we define the parabolic Chern class, in real cohomology, by
$${\parc}_k(E^*) =  P_k(c_1(E),\ldots c_r(E), \sum
\alpha^{(1)}_i[D_i],\ldots, \sum \alpha^{(r)}_i[D_i])
$$
We can unpack this formula with the help of the splitting
principle. For a parabolic line bundle $L^*$ with notation as in
example \ref{ex:paraline}, the weights of $L^*$ in the interval
$[0,1)$ are $(\beta_i)$. Then we see that the parabolic first Chern
class of $L^*$ is
\begin{equation}
  \label{eq:parc1}
  {\parc}_1(L^*)= c_1(L)+
  \sum_i\beta_i[D_i]. 
\end{equation}
Given a parabolic bundle $E^*$ of rank $r$, let $p:Fl(E)\to X$ denote
the full flag bundle of $E$. The pullback $p^*E$ carries a filtration
$F^i\subset E$ by subbundles such that associated graded
$G_i = F^i/F^{i+1}$ are line bundles. The parabolic structure on $E$
can be pulled back to a parabolic structure on $p^*E$ along $p^*D$,
and each $G_i$ carries the induced parabolic structure.  One sees from
above that:
\begin{lemma}\label{lemma:split}
  The parabolic Chern classes satisfy
$$1+ \sum p^*{\parc}_i(E^*)= \prod_i (1+\textnormal{par-}c_1(G_i^*))$$
\end{lemma}

\begin{lemma}[Biswas]\label{lemma:Chern}
  Given any parabolic vector bundle $E^*$ with weights in
  $\frac{1}{N}\Z$, let $\pi:Y\to X$ and $\E$ be the $G$-equivariant
  bundle corresponding to $E^*$ as in theorem \ref{thm:biswas}.  Then
$$\pi^*{\parc}_i(E^*)=c_i(\E)$$

\end{lemma}

\begin{proof}
  We can use lemma \ref{lemma:split} and the injectivity of the map
  $H^*(X, \R)\to H^*(Fl(E), \R)$ to reduce this to the case where
  $i=1$ and $E^*=L^*$ is a line bundle. Let us use $\mathcal{L}$
  instead of $\E$.  Then
  $$\pi^*
  {\parc}_1(L^*)=c_1(\pi^*L)+\sum_i\beta_i[\pi^*D_i]=c_1(\pi^*L)+\sum_ik_i\beta_iN[\tilde{D_i}].$$
  By Biswas \cite[(3.11)]{biswas},
  $$c_1(\mathcal{L})=c_1(\pi^*L)+\sum_i\beta_ik_iN[\tilde{D_i}]=\pi^*
  \textnormal{par-}c_1(L^*).$$
\end{proof}

\begin{lemma}\label{lemma:QapproxCh}
  Given $\epsilon>0$ and a parabolic bundle $E^*$ with trivial
  parabolic Chern classes, there exists a parabolic bundle ${E'}^*$
  with trivial parabolic Chern classes and rational weights which is
  $\epsilon$-close to $E^*$.
\end{lemma}

\begin{proof}
  We first treat the case where $E^*$ is a line bundle with normalized
  weight
  vector $\alpha= (\alpha_i)$.
    By the assumption, we have
$$\textnormal{par-}c_1(E^*)=c_1(E)+\sum_{j=1}^n\alpha_j\cdot[D_j]=0\ \in  H^2(X,\R)$$
Since $c_1(E)$ and $[D_j]$ are in $H^2(X, \Q)$, the above equation
defines a rational affine subspace of $\R^n$. The rational vectors are
dense in this subspace.  Therefore we can choose a rational vector
$\alpha'=(\alpha'_i)$ with $|{\alpha'}_j-\alpha_j|<\epsilon$ for all
$i$ and
$$c_1(E)+\sum_{j=1}^n\alpha'_j\cdot[D_j]=0\ \in  H^2(X,\R).$$ 

Now we do the general case. By the assumption that $\parc_i(E^*)=0$ in
$H^{2i}(X, \R)$, we have $p^*\parc_i(E^*)=0$ in $H^{2i}(Fl(E), \R)$.
By lemma \ref{lemma:split}, we know that
$\textnormal{par-}c_1(G_k^*)=0$ in $H^{2}(Fl(E), \R)$ for all $k$.
After identifying $\Comp(D)=\Comp(p^*D)$, we may identify weight
vectors of $E^*$ and $p^*E^*$.  It is easy to see that the normalized
weight vectors of $E^*$ are precisely the weight vectors of the
various $G^*_k$.  By the first paragraph, we can find $G'^*_k$,
$\epsilon$-close to $G^*_k$, having rational normalized weights, and
$\textnormal{par-}c_1(G'^*_k)=0$.  Let ${E'}^*$ be $E^*$ as a
quasi-parabolic bundle, but with the normalized weight vectors of
$G'^*_k$.

\end{proof}

\section{Stability and Semistability}

In this section, we will recall the definitions of (semi)stability for
parabolic and equivariant Higgs sheaves.  There are in fact two
different notions: $\mu$-, or slope, (semi)stability and $p$- , or
Hilbert polynomial, (semi)stability.  The $\mu$-(semi)stability
condition behaves well with respect to Biswas' correspondence, while
$p$-(semi)stability is more convenient for the construction of moduli
spaces.

We fix a very ample line bundle $H$ on $X$.  For a parabolic sheaf
$E^*$, we have the following numerical invariants defined by Maruyama
and Yokogawa \cite{my}.  The parabolic Hilbert polynomial of $E^*$ is
\begin{equation}
  \label{eq:Hilbpoly}
  \textnormal{par-}P_{E^*}(m):=\int_0^1P_{E^t}(m)dt
\end{equation}
where $P_{E^t}(m)$ is the Hilbert polynomial of $E^t$ with respect to
$H$. The normalized parabolic Hilbert polynomial of $E^*$ is
$\textnormal{par-}p_{E^*}(m):=\textnormal{par-}P_{E^*}(m)/\rank(E)$.
The parabolic degree of $E^*$ is defined to be
\begin{equation}
  \label{eq:pardeg}
  \textnormal{par-deg}(E^*):=\int_0^1\deg(E^t)dt+\rank(E)\cdot \deg
  (D)
\end{equation}
where $deg(E^t)$ is the usual degree of $E^t$, with respect to $H$.
The parabolic $H$-slope of $E^*$ is
$\textnormal{par-}\mu_H(E^*):=\textnormal{par-deg}(E^*)/\rank(E)$.

For a parabolic Higgs sheaf $(E^*, \theta)$, the above invariants are
defined to be that of its underlying parabolic sheaf $E^*$. We have
the following proposition.

\begin{prop}
  For any parabolic bundle $E^*$, we have
  $$\textnormal{par-deg}(E^*)=\textnormal{par-}c_1(E^*)\cdot H^{d-1}$$
  (We recall that $d=\dim X$.)
\end{prop}
\begin{proof}
  By taking top exterior powers, we may reduce to the case where
  $E^*=L^*$ is a line bundle.  We may assume that $L^*$ is as
  described in example \ref{ex:paraline} with
  $0\le \beta_1\le \ldots \le \beta_n <\beta_{n+1}=1$. In fact, we may
  assume that the $\beta_i$ form a strictly increasing sequence,
  because both sides of the expected formula depend continuously on
  these parameters.  Then by lemma \ref{lemma:paraline}, the
  normalized weights are given by $\beta_1,\ldots, \beta_n$ and the
  filtration by $F^i(L)= L(-D_1-\ldots- D_{i})$.

  By \eqref{eq:pardeg}, the left hand side of the purported equation
  is
  \begin{equation*}
    \begin{aligned}
      \textnormal{par-deg}(L^*)=&\sum_{i=0}^{n}\deg(F^i(L))\cdot(\beta_{i+1}-\beta_{i})+\deg(D)\\
      =&\sum_{i=0}^{n}(c_1(L)-\sum_{j=1}^{i}c_1(D_j))\cdot H^{d-1}\cdot(\beta_{i+1}-\beta_{i})+\deg(D)\\
      =&\deg(L)+\sum_{i=0}^n\beta_i\deg(D_i)
    \end{aligned}
  \end{equation*}
  
  By \eqref{eq:parc1}, the right hand side is
  \begin{equation*}
    \begin{aligned}
      \textnormal{par-}c_1(L^*)\cdot H^{d-1}=&(c_1(L)+ \sum_{j=1}^n\beta_j\cdot c_1(D_j))\cdot H^{d-1}\\
      =&\deg(L)+\sum_{i=0}^n\beta_i\deg(D_i)
    \end{aligned}
  \end{equation*}

\end{proof}

We can define (semi)stability of parabolic and $G$-equivariant Higgs
bundles using the numerical invariants defined above.

\begin{df} [{\cite{biswas2}\cite{my}}]\label{df:stable}

1) A parabolic Higgs sheaf $(E^*, \theta)$ on $X$ is called
$\mu_H$-stable (resp. $\mu_H$-semistable), or simply slope stable
(resp. semistable), if for any coherent
saturated subsheaf $V$ of $E$, with $0<\textnormal{rank}
V<\textnormal{rank} E$ and $\theta(V)\subseteq V\otimes
\Omega^1_X(\log D)$, the condition 
$$\textnormal{par-}\mu_H(V^*)< \textnormal{par-}\mu_H(E^*) \ 
(resp.\  \textnormal{par-}\mu_H(V^*)\leq \textnormal{par-}\mu_H(E^*))$$ 
is satisfied, where $V^*$ carries the induced the parabolic structure
from $E^*$, i.e. $V_{\alpha}:=V\cap E^{\alpha}.$ Slope
stability or semistability, with respect to  $\pi^*H$, for a $G$-equivariant Higgs sheaf
$(\mathcal{E}, \vartheta)$ on $Y$ is defined similarly, where  in addition $\mathcal{V}$
is required to be a $G$-equivariant subsheaf.

2) A parabolic Higgs sheaf $(E^*, \theta)$ on $X$ is called
$p$-stable (resp. $p$-semistable) if for any coherent saturated subsheaf $V$ of
$E$, with $0<\textnormal{rank} V<\textnormal{rank} E$, and
$\theta(V)\subseteq V\otimes \Omega^1_X(\log D)$, the condition 
$$\textnormal{par-}p_{V^*}(m)<\textnormal{par-}p_{E^*}(m) \ 
(resp.\  \textnormal{par-}p_{V^*}(m)\leq\textnormal{par-}p_{E^*}(m))$$
 is satisfied for all sufficiently large integers $m$, where $V^*$ carries
the induced parabolic structure. 
\end{df}

\begin{lemma}\label{twostable}
  A $\mu_H$-stable parabolic Higgs bundle $(E^*, \theta)$ on $(X, D)$
  is $p$-stable.
\end{lemma}

\begin{proof}
  Denote $H$ by $\OO_X(1)$. For any torsion free sheaf $E$, the
  Hilbert polynomial $P_{E}(m)=\dim H^0(X, E\otimes\OO_X(m))$, for
  $m\gg 0$. In particular, $P_{E}$ can be uniquely written in the form
$$P_{E}(m)=\sum_{i=0}^da_i(E)\frac{m^i}{i!}.$$ 
The rank of $E$ is $\rank (E)=\frac{a_d(E)}{a_d(\OO_X)}$. By
the Hirzebruch-Riemann-Roch formula, we have
$\deg (E)=a_{d-1}(E)-\rank(E)\cdot a_{d-1}(\OO_X)$.

Now we consider the parabolic Higgs bundle $(E^*, \theta)$. For the
parabolic Hilbert polynomial, we have
$$\textnormal{par-}P_{E^*}(m)=\sum_{i=0}^d(\sum_{j=0}^{l}a_i(F^{j}(E))(\alpha_{j+1}-\alpha_{j}))\frac{m^i}{i!},$$ 
where $\alpha_0=0$ and $\alpha_{l+1}=1$. Then the normalized parabolic
Hilbert polynomial of $(E^*, \theta)$ is
\begin{equation}\label{eq:parpE}
  \begin{aligned}
    \textnormal{par-}p_{E^*}(m)=\frac{a_d(\OO_X)}{a_d(E)}\sum_{i=0}^d(\sum_{j=0}^{l}a_i(F^{j}(E))(\alpha_{j+1}-\alpha_{j}))\frac{m^i}{i!}.
  \end{aligned}
\end{equation}

For the parabolic degree, we have
$$\textnormal{par-deg}(E^*)=\sum_{j=1}^{l}a_{d-1}(F^{j}(E))(\alpha_{j+1}-\alpha_{j})+\rank(E)\cdot (\deg(D)-a_{d-1}(\OO_X)).$$
Then the parabolic $H$-slope is
\begin{equation}\label{eq:parmuE}
  \begin{aligned}
    \textnormal{par-}\mu_H(E^*)=\frac{a_d(\OO_X)}{a_d(E)}\sum_{j=1}^{l}a_{d-1}(F^{j}(E))(\alpha_{j+1}-\alpha_{j})+\deg(D)-a_{d-1}(\OO_X).
  \end{aligned}
\end{equation}

Since the parabolic Higgs bundle $(E^*, \theta)$ is $\mu_H$-stable,
for any coherent subsheaf $V$ of $E$, satisfying the conditions of
definition \ref{df:stable}, we have
$\textnormal{par-}\mu_H(V^*)< \textnormal{par-}\mu_H(E^*)$, i.e.,
\begin{equation}\label{eq:ineqpE}
  \begin{aligned}
    \frac{a_d(\OO_X)}{a_d(V)}\sum_{j=1}^{l}a_{d-1}(F^{j}(V))(\alpha_{j+1}-\alpha_{j})<\frac{a_d(\OO_X)}{a_d(E)}\sum_{j=1}^{l}a_{d-1}(F^{j}(E))(\alpha_{j+1}-\alpha_{j}).
  \end{aligned}
\end{equation}

In general, for any parabolic Higgs sheaf $(G^*, \theta)$, by the
above equation \eqref{eq:parpE}, the leading term of
$\textnormal{par-}p_{G^*}(m)$ is always $\frac{m^d}{d!}$, and the
$d-1$ degree term
is
$$\frac{a_d(\OO_X)}{a_d(G)}\sum_{j=1}^{l}a_{d-1}(F^{j}(G))(\alpha_{j+1}-\alpha_{j})\frac{m^{d-1}}{(d-1)!}.$$
Hence we
get $$\textnormal{par-}p_{V^*}(m)<\textnormal{par-}p_{E^*}(m),$$ for
all sufficiently large integers $m$, by the previous inequality
\eqref{eq:ineqpE}.
\end{proof}

The $\mu$-(semi)stability condition behaves well in Biswas's
correspondence. In fact, we have the following result by Biswas.
\begin{lemma}[{Biswas \cite[theorem 5.5]{biswas2}}]\label{lemma:stab}
  Under Biswas's correspondence in theorem \ref{thm:biswas},
  $\mu_H$-semistable (Higgs) bundles with weights in $\frac{1}{N}\Z$
  correspond to $\mu_{\pi^*H}$-semistable $G$-equivariant (Higgs)
  bundles.
\end{lemma}

 Finally, we note that {\em a priori}  a $\mu_{\pi^*H}$-semistable $G$-equivariant Higgs
  bundle need  not be   $\mu_{\pi^*H}$-semistable in the usual
  sense. Fortunately, this is the case, and was observed by Biswas \cite[lemma 2.7]{biswas} in the
  vector bundle setting. The same reasoning applies here.

  \begin{lemma}\label{lemma:equivstable} For a $G$-equivariant
    Higgs bundle $(\mathcal{E}, \vartheta)$, if it is $\mu_{\pi^*H}$-semistable
    as a $G$-equivariant Higgs bundle, then the underlying Higgs
    bundle $(\mathcal{E}, \vartheta)$ is $\mu_{\pi^*H}$-semistable in the usual
    sense.
  \end{lemma}

\begin{proof}

  We fix the polarization $\mathcal{H}=\pi^*H$ on $Y$, which is a
  $G$-equivariant line bundle. By Simpson \cite[lemma 3.1]{simp2}
  (and the  paragraph following the lemma), there exists a canonical Harder-Narasimhan
  filtration by Higgs subsheaves
  $$0\subsetneq (\mathcal{E}_1, \vartheta|_{\mathcal{E}_1})\subsetneq
  (\mathcal{E}_2, \vartheta|_{\mathcal{E}_2})\subsetneq \ldots
  \subsetneq (\mathcal{E}_k, \vartheta|_{\mathcal{E}_k})=(\mathcal{E},
  \vartheta),$$ with strictly decreasing slope with respect to
  $\mathcal{H}$. This must be stable under $G$. So if $(\mathcal{E},
  \vartheta)$ were not semistable in the usual sense,
we would  have
  $$\mu_{\mathcal{H}}((\mathcal{E}_1,
  \vartheta|_{\mathcal{E}_1}))>\mu_{\mathcal{H}}((\mathcal{E},
  \vartheta)),$$ 
contradicting semistability in the equivariant sense.

\end{proof}

\section{Review of Nonabelian Hodge theory}\label{sect:nonabelian}

Natural examples of  parabolic Higgs bundles come from variations of Hodge structures.
Suppose that $(V_o, \nabla_o )$ is a flat bundle  underlying a 
polarized variation of Hodge structure on $X-D$ \cite{griffiths, ws} with
unipotent monodromy around components of $D$. Then we can form the
Deligne canonical extension $V$ to $V_o$.  The bundle $V_o$ also
carries  a Hodge filtration $F_o^\dt$ satisfying Griffiths' transversality.
By a theorem of Schmid
\cite{ws}, the Hodge filtration extends to a  filtration $F^\dt$ of $V$.
 Let $E = \Gr_FV$, and $\theta = \Gr_F\nabla$. Then, as observed
 already in \cite{arapura}, $(E, \theta)$
 is a Higgs bundle with trivial parabolic structure and trivial Chern classes. 
If the monodromies are quasi-unipotent,
as in geometric examples, then we may use a  Galois $G$-cover $\pi: Y
\rightarrow X$, as in section~\ref{sect:biswas}, such that $\pullback{\pi}\nabla^o$ is
unipotent. The  Higgs bundle associated to $\pullback{\pi}\nabla^o$ is
naturally $G$-equivariant, and thus via theorem~\ref{thm:biswas}, we
get a parabolic Higgs bundle on $X$ with rational weights and trivial
parabolic Chern classes.
To more general complex variations of Hodge structures, we can also
associate a parabolic Higgs bundles with vanishing parabolic Chern
classes (but real weights), but this relies on nonabelian Hodge
theory.  We review these ideas now, since they will be needed later.
Let us start with a definition of a  complex polarized variation of Hodge
structures or a $\C$-PVHS from the $C^\infty$ point of view. Let
$A^{i,j}(H)$ denote the space of $C^{\infty}$ $(i,j)$-forms with
values in a bundle $H$.

\begin{df}
  A complex polarized variation of Hodge structures over $U=X-D$ is a $C^{\infty}$-vector bundle $H$ with a
  decomposition $H={\bigoplus}_p H^p$, a flat connection $\mathcal{D}$
  and a horizontal indefinite Hermitian form $k_H$. These are required to satisfy
  Griffiths' transversality
$$\mathcal{D}: H^p\longrightarrow A^{0,1}(H^{p+1})\oplus
A^{1,0}(H^{p})\oplus A^{0,1}(H^{p}) \oplus A^{1,0}(H^{p-1}),$$
the decomposition ${\bigoplus}_p H^p$ is orthogonal with respect
to $k_H$, and $k_H$ is positive (negative) definite on $H^p$
with $p$ is even (odd). 
\end{df}

To relate this to the more traditional perspective, decompose
$\mathcal{D}$ into operators of types $(1,0)$ and $(0,1)$ 
$$
\mathcal{D}=\mathcal{D}^{1,0}+\mathcal{D}^{0,1}  
$$
The operator $\mathcal{D}^{0,1}$ defines a complex structure on $H$,
and let $V_o$ denote the corresponding holomorphic bundle.
The operator $\nabla=\mathcal{D}^{1,0}$ induces a holomorphic
connection on $V_o$, and $F^pV_o = H^p\oplus H^{p+1}\oplus \ldots$  forms
a holomorphic subbundle such that $\nabla(F^pV_0)\subset
\Omega_U^1\otimes F^{p-1}V_o$. The graded holomorphic bundle 
 $E_o = \Gr_FV_o$ carries a Higgs field $\theta =\Gr_F\nabla$.

The Higgs bundle $(E_o,\theta)$ can be constructed from a different point of view which is more general.
First observe that after changing signs of  $k_H$ on odd $H^p$, we get a positive
  definite Hermitian form $K_H$.
Suppose more generally that we are given a $C^{\infty}$ flat bundle $(H,\mathcal{D})$ with a Hermitian
metric $K$ over $U$, we can  decompose
$\mathcal{D}=\mathcal{D}^{1,0}+\mathcal{D}^{0,1}$ as above.
Let $\delta'$ and $\delta''$ be operators of type $(1,0)$ and $(0,1)$ 
such that $\mathcal{D}^{1,0}+\delta''$ and $\delta'+\mathcal{D}^{0,1}$ 
are metric connections with respect to the metric $K$, i.e.,

$$(\mathcal{D}^{0,1}u, v)_K+(u, \delta'v)_K=\mathcal{D}^{0,1}(u,v)_K, $$ 
$$(\delta''u, v)_K+(u, \mathcal{D}^{1,0}v)_K=\delta''(u,v)_K, $$ 
for all local sections $u,v$ of $ H$.
Define 
\begin{equation*}
  \begin{aligned}
\bar{\partial} &:=\frac{1}{2}(\mathcal{D}^{0,1}+\delta'')\\      
\theta &:=\frac{1}{2}(\mathcal{D}^{1,0}-\delta')\\
 \end{aligned}
\end{equation*}

\begin{df}
A triple $(H, \mathcal{D}, K)$  on $U$ is called a
\textit{harmonic bundle} if the pseudo-curvature $G_K
:=\bar{\partial}\theta=0$.
A harmonic bundle $(H, \mathcal{D}, K)$ is tame if the eigenvalues of
the associated Higgs field $\theta$ (which are multivalued $1$-forms) 
have poles of order at most 1, near the divisor $D$.
\end{df}

Given a harmonic bundle $(H, \mathcal{D}, K)$, $H$ equipped with
$\mathcal{D}^{0,1}$ becomes a holomorphic bundle $V_o$ over $U$ with a holomorphic
connection $\nabla$ induced from $\mathcal{D}^{1,0}$; $H$ equipped with 
$\bar{\partial}$ becomes a holomorphic bundle $E_o$ over $U$ and $\theta$
becomes a holomorphic Higgs field $\theta: E_o\rightarrow \Omega_U^1\otimes
E_o$.    If $(H, \mathcal{D}, K)$ is tame, then both $V_o$ and $E_o$
extend to  parabolic bundles over $X$ making the latter into a
parabolic Higgs bundle. 
 Roughly speaking, $E^{\alpha}\subset j_*E_o$ is
generated by sections $s$, with $|s(x)|_{K_H}=
O(|f(x)|^{\alpha-\epsilon})$, for all $\epsilon >0$, where $f$ is a
local equation for $D$. A similar description holds  for $V^\alpha$.
The connection $\mathcal{D}^{0,1}$ induces a
logarthmic connection on $V^\alpha$.

We have the following correspondence given by Simpson \cite[main theorem]{simp3}
for curves and Mochizuki \cite[theorem 1.4]{moch1} for higher
dimensional quasi-projective varieties.

\begin{thm}[Kobayashi-Hitchin Correspondence]\label{khcorr}
For the quasi-projective variety $X-D$ and ample line bundle $H$ over
$X$, we have a one to one correspondence between tame harmonic bundles 
$(H, \mathcal{D}, K)$ and $\mu_H$-polystable parabolic Higgs bundles 
$(E^*, \theta)$ with vanishing parabolic Chern classes, 
where ``$\mu_H$-polystable'' bundle means a direct sum of $\mu_H$-stable bundles.
\end{thm}

Mochizuki \cite[theorem 1.1]{moch2} gives a stronger statement, which
will not need.  
%
A key example of a tame harmonic is given by $\C$-PVHS.  This has
certain extra features as well.

\begin{prop} \label{vhsharmonic} Any $\C$-PVHS $(H, \mathcal{D}, k_H)$
  with  metric $K_H$ over $U$ is a tame harmonic bundle. The
  resulting parabolic   structure on $V$ agree  Deligne parabolic structure.
The filtration
$$F^pV^\alpha= V^\alpha\cap  j_* F^pV$$
gives a filtration by subbundles, and the associated graded $Gr_FV^\alpha$ can be
identified with $E^\alpha$.
\end{prop}

A proof can be found in  \cite[section 7]{brun}, although the result
is not explicitly stated in this form.  The proposition gives a
grading on $E$ by $Gr_F^p V^0$. 
Let 
\begin{equation}
  \label{eq:Fmax}
  F^{\max}V=Gr_F^p V^0 = F^pV^0
\end{equation}
where $p$ is the largest integer for which this is nonzero.
We will refer to this as the {\em smallest Hodge bundle} associated to
the variation.

\begin{cor}
  A $\C$-PVHS gives rise to a $\mu_H$-polystable parabolic Higgs
  bundle $(E^*, \theta)$ with vanishing parabolic Chern
  classes. Furthermore $\theta$ is nilpotent in the sense that it has zero
eigenvalues.  
\end{cor}

\begin{proof}
  The last statement follows from the fact that $\theta$ shifts the
  grading by $-1$.
\end{proof}

\bigskip

Next, we recall some facts about the moduli space of parabolic Higgs
sheaves and the Hitchin fibration from Yokogawa \cite{yokogawa1}.
Let $\Gamma$ denote the following data:
 a positive integer $m$, system of rational weights  $0<\alpha_1<\alpha_2<...<\alpha_l<1$ and polynomials 
$P, P_1, ..., P_l$. Consider the following contravariant functor:
$$\overline{\mathfrak{M}}(X, D,\Gamma): Sch/\mathbb{C}\longrightarrow
\mathcal{S}et$$
which assigns to any scheme $S$, the set of
isomorphism classes of flat families of  rank $m$ parabolic Higgs sheaves
$(E^*, \theta)$ over $(X\times S, D\times S)$ with the following
properties
\begin{enumerate}
\item For each closed point $s$ of $S$, 
$({E_s}^*, \theta_s):=(E^*, \theta)_s$  has weights $\alpha$ with
quasi-parabolic  structure $E_s\supsetneq F^1(E_s)\supsetneq \ldots\supsetneq F^l(E_s)\supsetneq E_s(-D\times\{s\})$.
 
\item  $({E_s}^*, \theta_s)$ is $p$-semistable.

\item The Hilbert polynomial of $E_s$ with respect to polarization $H$
  is $P$. The Hilbert polynomials of $E_s/{F^i(E_s)}$ are $P_i$.

\item  The parabolic Chern classes of $({E_s}^*, \theta_s)$ vanish. 

\end{enumerate}
Define an equivalence relation on $\overline{\mathfrak{M}}(X,
D,\Gamma)(S)$ by
 $(E^*, \theta)\sim (E'^*, \theta')$ if and only if there exists a
  line bundle $\mathcal{L}$ over $S$, such that $Gr^W_*(E^*, \theta)\cong
  Gr^{W'}_*(E'^*, \theta')\otimes\mathcal{L}$, where $W, W'$ are
  Jordan-H\"older filtrations (defined on \cite[p 457]{yokogawa1}).
Define
 $$\overline{\mathcal{M}}(X, D,\Gamma)(S) = \overline{\mathfrak{M}}(X, D,\Gamma)(S)/\sim$$
We denote by $\mathcal{M}(X, D,\Gamma)$ 
the subfunctor of $\overline{\mathcal{M}}(X, D,\Gamma)$ consisting  of all flat families of $p$-stable
parabolic Higgs bundles. Then we have the following theorem given by
Yokogawa \cite{yokogawa1}.

\begin{thm}[Yokogawa] \label{moduli} There exist quasiprojective  moduli spaces
  ${M}(X, D,\Gamma)\subset \overline{{M}}(X, D,\Gamma)$
coarsely representing  the functors
  $\mathcal{M}(X, D,\Gamma)$ and
  $\overline{\mathcal{M}}(X, D,\Gamma)$ respectively.
The closed points of  ${M}(X, D,\Gamma)$ are in one to one
  correspondence to the isomorphic classes of $p$-stable parabolic
  Higgs bundles $(E^*, \theta)$ of rank $m$ over $(X, D)$ with weights $\alpha$,
  Hilbert polynomials $\mathcal{P}$, and vanishing Chern classes.
\end{thm}
 
\begin{rmk}\label{rmk:Mpt}
Yokogawa states the second  result for a space of stable sheaves slightly larger than our ${M}(X,
  D,\Gamma) $.  Points of $\overline{M}(X, D,\Gamma)$ are $\sim$-equivalence
  classes of $p$-semistable sheaves on $X$.  
\end{rmk}

We will need to recall a few details of the construction.
Yokogawa \cite[\S 2]{yokogawa1} shows there is a scheme $R^{ss}$ on which a special linear
group $G$ acts, such that there are inclusions
$${M}(X, D,\Gamma)\subset\overline{{M}}(X,
D,\Gamma)\subset{R}^{ss}\sslash G$$ 
where the last space is the GIT
quotient. Let $R^0=R^{0}(X, D,\Gamma)$ denote the preimage of
$\overline{{M}}(X, D,\Gamma)$ in $R^{ss}$.  Then also by construction,
$X\times R^0$ comes with a family of parabolic sheaves inducing the
quotient map $R^0\to \overline{M}(X, D,\Gamma)$. We will refer to
this as the semi-universal sheaf.

Yokogawa has also generalized the construction and properties of the
Hitchin map of Simpson \cite{simp1,simp2} in the non-log case.

\begin{thm} [Yokogawa] \label{hitchin}
There is a  Hitchin map 
$$\mathfrak{h}: \overline{{M}}(X, D,\Gamma)\longrightarrow \mathfrak{V}(X, m):=
\bigoplus_{i=0}^{m-1} H^0(X, S^i\Omega^1_{X}(\log D)).$$
given by sending $(E^*, \theta)$ to its characteristic
polynomial. This map is projective.
\end{thm}

\begin{rmk}\label{rmk:nilp}
Note that $\mathfrak{h}(E^*, \theta)=0$ if and only if $\theta$ is
nilpotent.
\end{rmk}

\section{Vanishing Theorem: nilpotent case }

In this section, we prove a vanishing theorem for the de Rham complex of
any $\mu_H$-semistable parabolic Higgs bundle $(E^*, \theta)$,
 with vanishing parabolic Chern classes and nilpotent Higgs field $\theta$.

\begin{lemma}\label{lemma:perturbwts}
  Let $(E^*,\theta)$ be a $\mu_H$-semistable parabolic Higgs bundle. Then there exist $\epsilon >0$
  such that any parabolic Higgs bundle $\epsilon$-close to
  $(E^*,\theta)$ is $\mu_H$-semistable.
\end{lemma}

\begin{proof}
 
Suppose $(E^*, \theta)$ is stable. Let us denote the normalized weights
by $\{\alpha_1,...,\alpha_r\}$, and the quasiparabolic
structure by $E=F^0(E)\supsetneq F^1(E)\supsetneq \ldots \supsetneq
F^r(E)\supsetneq E(-D)$. 
Denote the degree of each bundle $F^i(E)$ by $d_i(E)$ and $d_i(V) =
\deg V\cap F^i(E)$ for any subsheaf $V\subseteq E$.
Suppose $V$ satisfies the conditions in definition \ref{df:stable}, then
$$\frac{\sum^r_{i=0}d_i(V)(\alpha_{i+1}-\alpha_i)}{\rank V}< \frac{\sum^r_{i=0}d_i(E)(\alpha_{i+1}-\alpha_i)}{\rank E}$$
or equivalently $$\sum^r_{i=0}(\rank V\cdot d_i(E)-\rank E\cdot d_i(V))(\alpha_{i+1}-\alpha_i)>0. $$
Since for any $V,\  \rank V\cdot d_i(E)-\rank E\cdot d_i(V)$ are
integers, we can find $\epsilon > 0$ such that
$$\sum^r_{i=0}(\rank V\cdot d_i(E)-\rank E\cdot d_i(V))(\alpha'_{i+1}-\alpha'_i)>0 $$ 
for  $|\alpha'_i-\alpha_i|<\epsilon$.

If $(E^*, \theta)$ is semistable, we have  the Jordan-H{\"o}lder
filtration of $(E^*, \theta)$ $$0\subset W_1\subset
W_2\subset...\subset (E^*, \theta)$$ 
such that the quotients $W_i/W_{i-1}$ are stable. We 
apply the above argument to these subquotients.
\end{proof}

\begin{prop}\label{prop:ration}
 Let $(E^*,\theta)$ be a $\mu_H$-semistable parabolic Higgs bundle with zero parabolic Chern classes. There exists
 a $\mu_H$-semistable parabolic Higgs bundle $({E'}^*,\theta')$ with  the same properties and rational weights such that $(E,\theta)=(E', \theta')$.
\end{prop}

\begin{proof}
This follows from lemma \ref{lemma:QapproxCh} and \ref{lemma:perturbwts}.

\end{proof}

We discuss the penultimate forms of the main result.
Given any parabolic Higgs bundle $(E^*,\theta)$, we have the associated de Rham complex
 $$\DR(E,\theta) = E\stackrel{\theta}{\to} \Omega_X^1(\log D)\otimes E\to
 \ldots.$$

\begin{thm}\label{thm:nilvan}
Let $(E^*,\theta)$ be a $\mu_H$-semistable parabolic Higgs bundle on
$(X, D)$ with vanishing Chern classes and with $\theta$ nilpotent. 
Let $L$ be an ample  line bundle on $X$.
Then 
$$
 \mathbb{H}^i(X,\DR(E,\theta)\otimes L)=0  
$$
for $i> d$, where $d=\dim X$.
\end{thm}

\begin{proof}
  For the $\mu_H$-semistable parabolic Higgs bundle $(E^*, \theta)$ on
  $X$, with trivial parabolic Chern classes, and $\theta$ nilpotent,
  by proposition \ref{prop:ration}, we can find a new
  $\mu_H$-semistable parabolic Higgs bundle $(E'^*, \theta)$ on $X$,
  with trivial parabolic Chern classes, $\theta$ nilpotent, and
  rational weights, such that $(E', \theta)=(E, \theta)$.  Thus we can
  assume that the weights of $(E^*, \theta)$ are in $\frac{1}{N}\Z^n$,
  for some integer $N$.

  Consider the Galois covering $\pi: Y\to X$ as described in
  section~\ref{sect:biswas}. By theorem \ref{thm:biswas}, \ lemma
  \ref{lemma:Chern}, and lemma \ref{lemma:stab}, we can find a
  $\mu_{\pi^*H}$-semistable $G$-equivariant Higgs bundle
  $(\mathcal{E}, \vartheta)$ on $Y$, such that $c_i(\mathcal{E})=0$
  and $\vartheta$ is nilpotent. The nilpotency of $\vartheta$ is
  coming from Biswas's construction of $\mathcal{E}$. Actually by
  Biswas \cite[(3.3)]{biswas},
  $\mathcal{E}\subset \pi^*(E\otimes\OO_X(D))$, over which $\theta$
  acts nilpotently. Also, by lemma \ref{lemma:equivstable}, we
    know that $(\mathcal{E}, \vartheta)$ is $\mu_{\pi^*H}$-semistable
    as a Higgs bundle.
Now we can apply the first main theorem of
  \cite{arapura} to conclude
$$\mathbb{H}^i(Y, \DR(\mathcal{E},\vartheta)\otimes \pi^*L)=0$$

Note that $Y\to X$ is a tower of cyclic covers, a standard calculation yields
\begin{equation}
  \label{eq:piOmega}
\Omega_Y^i(\log \tilde D)\cong \pi^*\Omega_X^i(\log  D)  
\end{equation}
So that  by the projection formula
\begin{equation}
  \label{eq:3}
  \begin{aligned}
    \pi_*(\Omega_Y^i(\log \tilde D)\otimes \E\otimes \pi^*L)^G &\cong
\Omega_X^i(\log  D)\otimes (\pi_*\E)^G \otimes L\\
&\cong \Omega_X^i(\log  D)\otimes E \otimes L\\
  \end{aligned}
\end{equation}
This can be seen to induce an isomorphism of complexes
$$\pi_*(\DR(\mathcal{E},\vartheta)\otimes \pi^*L)^G\cong \DR(E,
\theta)\otimes L$$
The next lemma shows that the natural map
$$\mathbb{H}^i(X, \DR(E, \theta)\otimes L)\to \mathbb{H}^i(Y, \DR(\mathcal{E},\vartheta)\otimes \pi^*L)^G$$
is an isomorphism.
\end{proof}

\begin{lemma}
  Given any cochain complex of complex vector spaces
  $(C^{\centerdot}, d)$ with a finite group $G$ action, there is a
  natural isomorphism
  $H^i((C^{\centerdot})^G)\cong H^i(C^{\centerdot})^G$.
\end{lemma}

\begin{proof}
  This follows from the exactness of the functor $(-)^G$ (Maschke's
  theorem).

\end{proof}

\begin{cor}\label{vhsvan1}
For a  Higgs bundle $(E^*, \theta)$ coming from a $\C$-PVHS 
we have 
$$\mathbb{H}^i(\DR(E, \theta)\otimes L)=0$$
for $i>d$.
\end{cor}

By the similar argument, we get  a stronger result. We refer to
\cite{lazarsfeld} for definitions of nef and ample vector bundles.

\begin{thm}\label{thm:nilvan2}
Let $(E^*,\theta)$ be a parabolic Higgs bundle satisfying the same
assumptions as theorem \ref{thm:nilvan}.
Let $M$ be a nef  vector bundle on $X$ such that $M(-\Delta)$ is ample
for some $\Q$-divisor supported on $D$ with coefficients  in $[0,1)$.
Then 
$$
 \mathbb{H}^i(X,\DR(E,\theta)\otimes M(-D))=0  
$$
for $i\ge  d+\rank M$, where $d=\dim X$.
\end{thm}

\begin{proof}
This is essentially the same argument, and we explain the necessary
modifications.  We  may assume rational weights with  $\pi:Y\to X$ as above.
The hypotheses of the theorem implies that for all $m>0$,  $\pi^*M^{\otimes
  m}(-\pi^*\Delta)$, and therefore $S^m(\pi^*M)(-\pi^*\Delta)$ is ample for all $m>0$.
Thus $\pi^*M(-\frac{1}{m}\pi^*\Delta)$ is ample for all $m$. For $m\gg 0$, the
coefficients of $\frac{1}{m}\pi^*\Delta$ lie in $[0,1)$.
So we can apply \cite[theorem 3]{arapura} (in place of \cite[theorem 1]{arapura} above)
to conclude that
\begin{equation}
  \label{eq:vanDR}
\mathbb{H}^i(Y, \DR(\mathcal{E},\vartheta)(-D)\otimes \pi^*M)=0  
\end{equation}
The dual of \eqref{eq:piOmega} is
$$\Omega_Y^i(\log \tilde D)(-\tilde D)\cong \pi^*\Omega_X^i(\log  D)(-D)  $$
From this, we obtain
$$\pi_*(\DR(\mathcal{E},\vartheta)(-\tilde D)\otimes \pi^*M)^G\cong \DR(E,
\theta)(-D)\otimes M$$
Putting this together with \eqref{eq:vanDR} proves the theorem.
\end{proof}

\begin{cor}\label{vhsvan}
For a  Higgs bundle $(E^*, \theta)$ coming from a $\C$-PVHS 
we have 
$$\mathbb{H}^i(\DR(E, \theta)\otimes M(-D))=0$$
for $i>d$.
\end{cor}

\begin{rmk}\label{rmk:2implies1}
  By taking $M=L(D)$ with $L$ an  ample line bundle, we see
  that theorem \ref{thm:nilvan2} implies theorem \ref{thm:nilvan} by
  choosing $\Delta= (1-\epsilon)D$, with $0<\epsilon\ll
  1$. Nevertheless, it seemed  clearer to state  and prove them
  separately. 
\end{rmk}

\section{Vanishing theorem: general case }

In this section, we will use the results in previous sections to prove
the vanishing theorem for semistable parabolic Higgs bundles with
vanishing parabolic Chern classes. 

\begin{lemma}\label{lemma:sstos}
Let $M$ be a vector bundle on $X$. 
Let $(E_1^*, \theta_1)\sim (E_2^*,\theta_2)$ be equivalent $p$-semistable
bundles with $(E_1, \theta_1)$ $p$-polystable (a direct sum of
$p$-stable bundles).
If 
$$\mathbb{H}^i(\DR(E_1,\theta_1)\otimes M)=0$$
then 
$$\mathbb{H}^i(\DR(E_2,\theta_2)\otimes M)=0$$

\end{lemma}
\begin{proof}
The assumptions say that $(E_1,\theta_1)\cong Gr^J_*(E_2,\theta_2)$, where
 $J$ is a  Jordan-H\"older filtration on $(E_2,\theta_2)$, and that
$$\mathbb{H}^i(\DR(Gr^{J}_*(E_2,\theta_2))\otimes M)=0$$
The conclusion follows easily from the exact sequences
$$0\to \DR(J_{i-1}(E_2, \theta_2))\to \DR(J_{i}(E_2,\theta_2))\to \DR(Gr_i^J(E_2,\theta_2))\to 0$$
and induction.
\end{proof}

The following is the main theorem.

\begin{thm} \label{stablevan}
Let $(E_*, \theta)$ be a $\mu_H$-semistable parabolic Higgs bundle on $(X,D)$
with vanishing parabolic Chern classes.
Let $M$ be a nef  vector bundle on $X$ such that $M(-\Delta)$ is ample
for some $\Q$-divisor supported on $D$ with coefficients  in $[0,1)$.
Then
$$\mathbb{H}^i(X,\DR(E,\theta)\otimes M(-D))=0$$ for $i>d+\rank M$.
\end{thm}

\begin{proof} By lemmas \ref{lemma:sstos} and \ref{twostable}, it suffices assume that
  $(E^*, \theta)$ is  $\mu_H$-stable.
By proposition \ref{prop:ration}, there is no loss of generality in assuming that
$(E^*, \theta)$ has rational weights $0<\alpha_1<\alpha_2<...<\alpha_l<1$ 
Let $m= \rank E$, and $E\supsetneq F^1(E)\supsetneq ...\supsetneq
 F^l(E)\supsetneq E(-D)$ denote the quasi-parabolic structure.
 Denote the Hilbert polynomials of $E$ with respect to $H$ by $P$, 
and the Hilbert polynomials of $E/{F^i(E)}$ by $P_i$, respectively,
and let $\Gamma= (m, \alpha_i, P, P_i)$.
Then by theorem \ref{moduli}, the isomorphism class $(E^*, \theta)$
can be regarded as a closed point, which is denoted by $p$, in the
moduli space 
${M}(X, D,\Gamma)\subset \overline{{M}}(X, D,\Gamma)$ described
earlier.  Now we consider the Hitchin map $\mathfrak{h}:
\overline{{M}}(X, D,\Gamma)\to\mathfrak{V}(X, m)$.
 We can assume that $\mathfrak{h}(p)\neq 0$, otherwise $\theta$ is
 nilpotent and we are done by theorem \ref{thm:nilvan}. 
Let $\C\cdot \mathfrak{h}(p)$ be the complex affine line passing
through $0$ and $\mathfrak{h}(p)$ in $\mathfrak{V}(X, m)$. Also, 
considering the $\C^*$-action
$$t: \overline{M}(X, D,\Gamma)\to \overline{{M}}(X, D,\Gamma)$$ 
$$(E^*, \theta)\mapsto (E^*, t\theta)$$ 
we will get a $\C^*$-orbit $C_p$ of the point $p$ as a curve in 
$\overline{{M}}(X, D,\Gamma)$. By properness of the Hitchin map, 
i.e., theorem \ref{hitchin}, we can extend the curve $C_p$ to
$\overline{C}_p$ by adding a point 
$p_0$ in the fiber $\overline{{M}}(X, D,\Gamma)_0$ over 0 of the Hitchin fibration. 

Now by earlier  remarks, we get the following commutative diagram

$$
\xymatrix{
\widetilde{C}_p \ar[d] \ar[r] &{{R}^0}(X, D,\Gamma)\ar[d]^{pr} & \tilde{p} \ar@{|->}[d]\\ 
\overline{C}_p \ar[d] \ar[r] &\overline{{M}}(X, D,\Gamma)\ar[d]^{\mathfrak{h}} & p \ar@{|->}[d]\\
\C\cdot \mathfrak{h}(p) \ar@{^{(}->}[r]&\mathfrak{V}(X, m) &\mathfrak{h}(p)}
$$
 where $\widetilde{C}_p$ is some curve in  ${R}^0(X, D,\Gamma)$
 mapping finitely to $\overline{C}_p$. Let
 $\tilde{p},\tilde{p}_0\in \widetilde{C}_p$ lie over  $p$ and $p_0$ respectively.
Now pullback the semi-universal parabolic bundle to the curve
$\widetilde{C}_p$ 
to obtain a parabolic bundle 
$(\mathcal{E^*, \vartheta})$ over $X\times \widetilde{C}_p$ which
is flat over $\widetilde{C}_p$. 
Note that for any point $q \in \widetilde{C}_p$ not lying over $p_0$,
the parabolic Higgs bundle $(\mathcal{E}^*, \vartheta)_{q}$ is
equivalent (under $\sim$) to $(E^*, t\theta)$, for some $t\in \C^*$, and is therefore stable.
Consequently, $(\mathcal{E}^*, \vartheta)_{q}\cong (E^*, t\theta)$.

Now consider the parabolic Higgs bundle $({E_0}^*, \theta_0)$ 
 over $X$ corresponding to ${p}_0$. 
This is  a fixed point for
the $\C^*$-action, so   by Mochizuki \cite[proposition
 1.9]{moch1} it must, in fact, come from  a $\C$-PVHS.  
Hence by corollary \ref{vhsvan}, we have 
$$\mathbb{H}^i(\DR(({E_0}, \theta_0))\otimes M(-D))=0, \
\textnormal{for}\ i>d.$$  
By proposition \ref{vhsharmonic} and theorem \ref{khcorr}, 
we have that $({E_0}^*, \theta_0)$ is $\mu_H$-polystable. Hence it is
parabolic $p$-polystable by 
lemma \ref{twostable}. Since $(\mathcal{E}^*,
\vartheta)_{\tilde{p}_0}\sim ({E_0}^*, \theta_0)$, lemma
\ref{lemma:sstos} and corollary~\ref{vhsvan} implies 
$$\mathbb{H}^i(\DR((\mathcal{E}, \vartheta)_{\tilde{p}_0})\otimes M(-D))=0, \ \textnormal{for}\ i>d.$$
Since
$$q\mapsto \dim \mathbb{H}^i(\DR((\mathcal{E}, \vartheta)_{q})\otimes
M(-D)) $$
is upper semi-continuous,
$$\mathbb{H}^i(\DR((\mathcal{E},\vartheta)_{q})\otimes M(-D))=0$$
for $i>d$, and $q$ in a small open nighborhood of $\tilde{p}_0$ in $\widetilde{C}_p$. Thus we get
$$\mathbb{H}^i(\DR(E, t\theta)\otimes M(-D))=0$$
for $i>d$ and $t$ small enough. This implies 
$$\mathbb{H}^i(\DR(E, \theta)\otimes M(-D))=0$$
for $i>d$.

\end{proof}

\begin{cor}\label{cor:main}
  If $L$ be an ample line bundle over $X$,
then 
$$\mathbb{H}^i(\DR(E, \theta)\otimes L)=0$$
for $i>d$.
\end{cor}

\begin{proof}
This follows from remark \ref{rmk:2implies1}.
  
\end{proof}

\begin{rmk}\label{rmk:reductiontoVHS}
A second proof of this corollary can be obtained  by redoing the proof of the theorem
with corollary \ref{vhsvan1} in place of corollary \ref{vhsvan}.
Although this still ultimately hinges on
\cite[theorem 1]{arapura} which was proved by characteristic $p$
methods, it  is possible to give an entirely characteristic
$0$ proof with the above trick as follows. The deformation argument used
above  shows that it suffices to prove
$$\mathbb{H}^i(\DR(E,\theta)\otimes L)=0, \quad i>d$$
 when $(E,\theta)$ comes from a $\C$-PVHS
$H$. Since $H\oplus \bar H$ is an $\R$-PVHS, we are reduced
to proving vanishing in this case. This can be done {\em in principle} by adapting
Schnell's proof of Saito's vanishing \cite{schnell} to the category of pure $\R$-Hodge
modules introduced in \cite{saito}.
\end{rmk}

\section{Semipositivity}

As an application of the vanishing theorem, we can obtain a
semipositivity theorem in the spirit the 
Fujita-Kawamata theorem. In fact, it
was inspired by the fairly recent semipositivity results of  Brunebarbe \cite{brun}
and  Popa-Schnell \cite[theorem
47]{popa}.

Given two parabolic Higgs bundles $(E^*,\theta)$ and $(G^*, \psi)$, their tensor product
becomes a parabolic bundle with filtration
$$(E^*\otimes G^*)^\alpha = \sum_{\beta+ \gamma=\alpha}
E^\beta\otimes G^\gamma$$
and Higgs field $\theta\otimes I_G + I_E\otimes \psi$. It is possible
for $(E^*\otimes G^*)^0\supsetneqq E\otimes G$. However,  there is an evident
criterion for equality.

\begin{lemma}\label{lemma:tensor}
  We have $(E\otimes G)^0= E\otimes G$ if the only solution to
  $\beta+\gamma=0$ is $\beta=\gamma=0$, where $\beta$ and $\gamma$ are weights of $E$ and $G$.
\end{lemma}

We can also define the symmetric powers of $(E^*,\theta)$ as a
quotient of the tensor power by the symmetric group.

\begin{cor}\label{cor:sym}
  Suppose that the weights of $(E^*,\theta)$ satisfy 
$$\alpha_{i_1}+\ldots +\alpha_{i_n}=0 \Rightarrow \alpha_{i_1}=\ldots=
\alpha_{i_n}=0$$
Then $S^n(E^*)^0= S^n(E)$.
\end{cor}

Given two tame harmonic bundles their tensor product carries a tame
harmonic metric, and this is compatible with tensor products of the
parabolic Higgs bundles. These facts are summarized in the proof of
\cite[corollary 5.18]{moch2}, although the details appear in \cite{moch15}.
Therefore by combining this with  theorem \ref{khcorr}, we obtain:

\begin{prop}\label{prop:tensor}
  If $(E^*,\theta)$ and $(G^*, \psi)$ are $\mu_H$-polystable Higgs
  bundles with trivial
  parabolic Chern classes, then their tensor product and
  symmetric powers have the same properties.
\end{prop}

\begin{thm}\label{thm:semipos}
  Suppose that $(E,\theta)$ is a $\mu_H$-polystable parabolic Higgs
  bundle on $(X,D)$ with vanishing parabolic Chern classes, and that there is a decomposition
  $E=E_+\oplus E_-$ 
  such that $\theta(E)\subseteq \Omega_X^1(\log D)\otimes E_-$.  
Let $L$ be a nef  line bundle on $X$ such that $L(-\Delta)$ is ample
for some $\Q$-divisor supported on $D$ with coefficients  in $[0,1)$.
Then
$$H^i(X,\omega_X\otimes E_+\otimes L)=0$$
for $i>0$. Furthermore, $E_+$ is nef.
\end{thm}

\begin{proof}
  The assumptions imply that
$$\omega_X\otimes E_+\subseteq \DR(E,\theta)$$
is a direct summand.  Therefore, there is an injection
$$H^i(X,\omega_X\otimes E_+\otimes L)\to
\mathbb{H}^i(X,\DR(E,\theta)\otimes L)$$
Therefore Theorem \ref{stablevan} implies the vanishing statement.

For some $0<\epsilon \ll 1$, we can replace $E$ with an 
$\epsilon$-close parabolic  Higgs bundle with generic 
weights.  Specifically, generic means that the conditions of corollary
\ref{cor:sym} hold for all $n$.
This ensures that
$S^n(E^*)^0= S^n(E)$. Proposition \ref{prop:tensor} imply
that this is $\mu_H$-polystable with trivial parabolic Chern classes.
Furthermore, we get a decomposition as above with $(S^nE)_+ = S^n(E_+)$.
Therefore
$$H^i(X,\omega_X\otimes S^n(E_+)\otimes L)=0$$
for all $n>0$ by the first part of the theorem.
Now apply \cite[lemma 3.1]{arapura} to conclude that $E_+$ is nef.
\end{proof}

We refer to  Viehweg \cite{viehweg} for the definition and basic properties of
weak positivity.

\begin{cor}
Let $V_o$ be $\C$-PVHS on a smooth Zariski open subset $U$ of a projective
variety $Z$. Then  the smallest Hodge bundle $F^{\max } V_o$ extends to
a  torsion free sheaf  $F$  over $Z$ which is weakly positive over $U$. 
\end{cor}

\begin{proof}
  We can choose a resolution of singularities $p:X\to Z$ which is an
  isomorphism over $U$ and such that $D=X-U$ has simple normal
  crossings. We have an extension $F^{\max } V =Gr^{\max }_FV$
  described  in \eqref{eq:Fmax}, which is nef by the above theorem.
Set $F= p_*F^{\max } V$. This is weakly positive over $U$ by \cite[lemma 1.4]{viehweg}.
\end{proof}

\begin{rmk}
  One can see that $F$ is independent of the choice of resolution.
\end{rmk}

\end{document}